\documentclass[11pt,a4paper]{article}

\usepackage[left=2cm,right=2cm,top=2cm,bottom=2cm]{geometry}
\usepackage{array}
\usepackage[fleqn]{amsmath}
\usepackage{amssymb,latexsym}
\usepackage{amsthm}
\usepackage[colorlinks=true,linkcolor=black,citecolor=black,urlcolor=black]{hyperref}
\usepackage{enumitem}
\usepackage{graphicx}
\usepackage{tikz}
\usepackage{todonotes}
\usepackage{marginnote}

\newenvironment{config}{\tt\small\setlength{\parskip}{1pt}}{}

\newtheorem{theorem}{Theorem}[section]

\newtheorem{proposition}[theorem]{Proposition}

\title{Fragments in symmetric configurations with block size 3}

\author{
Grahame Erskine\thanks{Open University, Milton Keynes, UK}\footnotemark[1]\\ \texttt{\small grahame.erskine@open.ac.uk}
\and Terry Griggs\footnotemark[1]\\ \texttt{\small terry.griggs@open.ac.uk}
\and Jozef \v{S}ir\'a\v{n}\thanks{Open University, Milton Keynes, UK and Slovak University of Technology, Bratislava, Slovakia}\footnotemark[2]\\ \texttt{\small jozef.siran@open.ac.uk}
}

\date{}

\begin{document}
\maketitle
\let\thefootnote\relax
\let\thefootnote\relax

\vspace*{-8ex}
\begin{abstract}
We begin the study of collections of three blocks which can occur in a symmetric configuration with block size 3, $v_3$. Formulae are derived for the number of occurrences of these and it is shown that the triangle, i.e. abf, ace, bcd is a basis. It is also shown that symmetric configurations without triangles exist if and only if $v=15$ or $v \geq 17$. Such configurations containing ``many'' triangles are also discussed and a complete analysis of the triangle content of those with a cyclic automorphism is given.
\end{abstract}

% ================================================================
\section{Introduction}\label{sec:intro}
In this paper we will be concerned with symmetric configurations with block size 3. First we recall the definitions. A \emph{configuration} $(v_r,b_k)$ is a finite incidence structure with $v$ points and $b$ blocks, with the property that there exist positive integers $k$ and $r$ such that
\begin{enumerate}[label=(\roman*),topsep=0pt,itemsep=0pt]
	\item each block contains exactly $k$ points;
	\item each point is contained in exactly $r$ blocks; and
	\item any pair of distinct points is contained in at most one block.
\end{enumerate}

If $v=b$ (and hence necessarily $r=k$), the configuration is called \emph{symmetric} and is usually denoted by $v_k$. We are interested in the case where $k=3$. The blocks will also be called \emph{triples}. 
A configuration is said to be \emph{decomposable} or \emph{disconnected} if it is the union of two configurations on distinct point sets. We are primarily interested in indecomposable (connected) configurations, and so unless otherwise noted, this is assumed throughout the paper.

It is natural to associate two graphs with a symmetric configuration $v_3$. The first is the \emph{Levi graph} or \emph{point-block incidence graph}, obtained by considering the $v$ points and $v$ blocks of a configuration as vertices, and including an edge from a point to every block containing it. It follows that the Levi graph is a cubic (3-regular) bipartite graph of girth at least six. The second graph is the \emph{incidence graph}, obtained by considering only the points as vertices and joining two points by an edge if and only if they appear together in some block. Thus the incidence graph is regular of valency 6 and order~$v$.

The above definition of configuration is the classical one, going back to the nineteenth century and the work of mathematicians such as Kantor~\cite{K} and Martinetti~\cite{M}. It is the meaning of the term in the section of the Handbook of Combinatorial Designs by Gropp~\cite{G}. However, in the last century the term came to be used more widely to mean any (small) collection of blocks which may appear in a combinatorial structure. This is the sense in which the term is used for example in chapter 13 of Triple Systems by Colbourn \&~Rosa~\cite{CR}. For the present paper this double usage of the term is particularly unfortunate as we wish to study the occurrence of (20th century) configurations in (19th century) configurations. Therefore, for this reason we will define any collection of blocks or triples which may appear in a symmetric configuration $v_3$ as a \emph{fragment}.

We begin by developing formulae for the number of occurrences of one-, two- and three-block fragments in symmetric configurations $v_3$. A one-block fragment is of course just a single triple and there are $v$ of these by definition. There are two two-block fragments: $A_1$, a pair of disjoint triples and $A_2$, a pair of intersecting triples. Here, and below, we follow the terminology used in~\cite{CR}. It is elementary to derive the corresponding formulae. Trivially $a_2 = 3v$ and so $a_1 = v(v-1)/2 -3v = v(v-7)/2$. Both of these formulae are called \emph{constant} meaning that they depend only on $v$ and so do not vary over all configurations $v_3$, irrespective of the structure of the individual configurations.  Other fragments are called \emph{variable}. 

There are five three-block fragments. Omitting set brackets and commas for simplicity these are $B_1$, abc, def, ghi (3-partial parallel class or 3-PPC); $B_2$, abc, def, dgh (hut); $B_3$, abc, ade, afg (3-star); $B_4$, abc, cde, efg (3-path); $B_5$, abf, ace, bcd (triangle). These are illustrated in Figure~\ref{fig:fragments}.

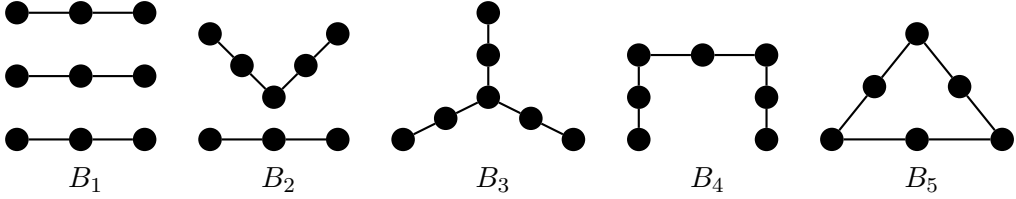
\begin{figure}\centering
\begin{tabular}{ccccc}
	\begin{tikzpicture}[x=0.2mm,y=-0.2mm,inner sep=0.2mm,scale=0.7,thick,vertex/.style={circle,draw,minimum size=8,fill=black}]
\node at (260,280) [vertex] (v1) {};
\node at (320,280) [vertex] (v2) {};
\node at (380,280) [vertex] (v3) {};
\node at (260,340) [vertex] (v4) {};
\node at (320,340) [vertex] (v5) {};
\node at (380,340) [vertex] (v6) {};
\node at (260,400) [vertex] (v7) {};
\node at (320,400) [vertex] (v8) {};
\node at (380,400) [vertex] (v9) {};
\path
	(v1) edge (v2)
	(v2) edge (v3)
	(v4) edge (v5)
	(v5) edge (v6)
	(v7) edge (v8)
	(v8) edge (v9)
	;
\end{tikzpicture} &
	\begin{tikzpicture}[x=0.2mm,y=-0.2mm,inner sep=0.2mm,scale=0.7,thick,vertex/.style={circle,draw,minimum size=8,fill=black}]
\node at (260,300) [vertex] (v1) {};
\node at (290,330) [vertex] (v2) {};
\node at (320,360) [vertex] (v3) {};
\node at (350,330) [vertex] (v4) {};
\node at (380,300) [vertex] (v5) {};
\node at (260,400) [vertex] (v7) {};
\node at (320,400) [vertex] (v8) {};
\node at (380,400) [vertex] (v9) {};
\path
	(v1) edge (v2)
	(v2) edge (v3)
	(v4) edge (v5)
	(v7) edge (v8)
	(v8) edge (v9)
	(v3) edge (v4)
	;
\end{tikzpicture} &
	\begin{tikzpicture}[x=0.2mm,y=-0.2mm,inner sep=0.2mm,scale=0.7,thick,vertex/.style={circle,draw,minimum size=8,fill=black}]
\node at (320,300) [vertex] (v1) {};
\node at (320,340) [vertex] (v2) {};
\node at (320,380) [vertex] (v3) {};
\node at (360,400) [vertex] (v4) {};
\node at (400,420) [vertex] (v5) {};
\node at (240,420) [vertex] (v7) {};
\node at (280,400) [vertex] (v8) {};
\path
	(v1) edge (v2)
	(v2) edge (v3)
	(v4) edge (v5)
	(v7) edge (v8)
	(v3) edge (v4)
	(v3) edge (v8)
	;
\end{tikzpicture} &
	\begin{tikzpicture}[x=0.2mm,y=-0.2mm,inner sep=0.2mm,scale=0.7,thick,vertex/.style={circle,draw,minimum size=8,fill=black}]
\node at (380,340) [vertex] (v1) {};
\node at (320,340) [vertex] (v2) {};
\node at (260,340) [vertex] (v3) {};
\node at (380,380) [vertex] (v4) {};
\node at (380,420) [vertex] (v5) {};
\node at (260,420) [vertex] (v7) {};
\node at (260,380) [vertex] (v8) {};
\path
	(v1) edge (v2)
	(v2) edge (v3)
	(v4) edge (v5)
	(v7) edge (v8)
	(v3) edge (v8)
	(v1) edge (v4)
	;
\end{tikzpicture} &
	\begin{tikzpicture}[x=0.2mm,y=-0.2mm,inner sep=0.2mm,scale=0.7,thick,vertex/.style={circle,draw,minimum size=8,fill=black}]
\node at (360,330) [vertex] (v1) {};
\node at (320,280) [vertex] (v2) {};
\node at (280,330) [vertex] (v3) {};
\node at (400,380) [vertex] (v4) {};
\node at (320,380) [vertex] (v7) {};
\node at (240,380) [vertex] (v8) {};
\path
	(v1) edge (v2)
	(v2) edge (v3)
	(v7) edge (v8)
	(v3) edge (v8)
	(v1) edge (v4)
	(v4) edge (v7)
	;
\end{tikzpicture} \\
	$B_1$ & $B_2$ & $B_3$ & $B_4$ & $B_5$
\end{tabular}
\caption{The possible three-block fragments}
\label{fig:fragments}
\end{figure}

Let the number of occurrences of fragment $B_i$ be $b_i$, $1 \leq i \leq 5$.  The number of triangles in a symmetric configuration $v_3$ is equal to the number of 6-cycles in its Levi graph and is therefore variable. Let $b_5 = t$. Trivially $b_3 = v$. To determine $b_4$, consider a pair of intersecting triples and the eight three-block fragments formed by this pair of triples and the third block through any of the four points other than the point of intersection. Then $3v \times 8 = 6t + 2b_4$ from which $b_4 = 3(4v - t)$. For $b_2$, again consider a pair of intersecting triples and a third block through any other point. Then $3v(v - 5) \times 3 = 3t + 4b_4 + 6b_3 + 3b_2$ from which $b_2 =  3(v(v - 11) + t)$. Finally for $b_1$, consider a pair of disjoint triples and a third block through any other point. Then $v(v - 7)(v - 6) \times 3 /2 = b_4 + 4b_2 + 9b_1$ from which $b_1 = (v^3 - 21v^2 + 122v - 6t)/6$. For ease of reference we collect these formulae together in a table.
\begin{center}
\begin{tabular}{l}
     $b_1 = (v^3 - 21v^2 + 122v - 6t)/6$,\\
     $b_2 =  3(v(v - 11) + t)$,\\
     $b_3 = v$,\\
     $b_4 = 3(4v - t)$,\\ 
     $b_5 = t$.
\end{tabular}
\end{center}

From the above, $0 \leq t \leq 4v$. It is natural to ask whether symmetric configurations $v_3$ with values of $t$ at either end of the spectrum exist and how these may be constructed. At the lower end, configurations with no triangles exist for $v=15$ and all $v \geq 17$. This is proved in Section~\ref{sec:notri}. At the upper end, the unique configuration $7_3$, which is the Fano plane, contains 28 triangles and is the only configuration attaining the upper bound. Configurations with many triangles are studied in Section~\ref{sec:manytri}. Section~\ref{sec:cyclic} deals with cyclic configurations, i.e. those with a cyclic automorphism. Finally in Section~\ref{sec:conclu} we discuss some further investigations suggested by the work in this paper. Computer results on the distribution of the values of $t$ for all configurations $v_3$ for $7 \leq v \leq 16$ are given in Table~\ref{tab:triangles} in the Appendix at the end of the paper. 

% ================================================================
\section{Configurations with no triangles}\label{sec:notri}
As stated above, the number of triangles in a symmetric configuration $v_3$ is equal to the number of 6-cycles in its Levi graph. Therefore in order to construct configurations with no triangles it is both necessary and sufficient to construct cubic bipartite graphs with girth equal to 8 (or greater than or equal to 8). We do this recursively, using the following result.

\begin{proposition}\label{vplus10}
Let $G$ be a cubic bipartite graph of order $2v$, of girth 8 and which also contains a 10-cycle. Then there exists a cubic bipartite graph $G^*$ of order $2v+10$, also of girth 8 and containing a 10-cycle.
\end{proposition}
\begin{proof}
Consider the graph $G$. Denote the vertices of the 10-cycle by $v_0,v_1,\ldots,v_9$. Let the third vertex to which the vertex $v_i$ is adjacent be $w_i,~0 \leq i \leq 9$. The vertices $w_i$ are distinct. Further the only two such vertices which can be adjacent are $\{w_i, w_{i+5}\},~0 \leq i \leq 4$. Delete the five edges $\{v_{2i-1},v_{2i}\},~1 \leq i \leq 5$, subscripts modulo 10, to form a graph $G'$. Now construct a graph $H$ consisting of a 10-cycle $u_0,u_3,u_4,u_7,u_8,u_1,u_2,u_5,u_6,u_9$ where in addition each vertex $u_i$ is adjacent to a vertex~$v_i$. Identify the vertices $v_i$ in the two graphs $G'$ and $H$ to form the graph $G^*$. Clearly the graph is cubic, bipartite (as represented by the black/white colouring of the vertices as shown in Figure~\ref{fig:girth8}), of order $2v+10$ and contains a 10-cycle. It remains to prove that it has girth 8.

\begin{figure}\centering
    \begin{tabular}{cc}
        \begin{tikzpicture}[x=0.2mm,y=-0.2mm,inner sep=0.2mm,scale=1,thick,vertex/.style={circle,draw,minimum size=8,fill=lightgray}]
\node at (324,227) [vertex,fill=black,label=above left:{$v_0$}] (v1) {};
\node at (416,227) [vertex,fill=white,label=above right:{$v_1$}] (v2) {};
\node at (491,282) [vertex,fill=black,label=above right:{$v_2$}] (v3) {};
\node at (520,370) [vertex,fill=white,label=right:{$v_3$}] (v4) {};
\node at (491,458) [vertex,fill=black,label=below right:{$v_4$}] (v5) {};
\node at (416,513) [vertex,fill=white,label=below right:{$v_5$}] (v6) {};
\node at (324,513) [vertex,fill=black,label=below left:{$v_6$}] (v7) {};
\node at (249,458) [vertex,fill=white,label=below left:{$v_7$}] (v8) {};
\node at (220,370) [vertex,fill=black,label=left:{$v_8$}] (v9) {};
\node at (249,282) [vertex,fill=white,label=above left:{$v_9$}] (v10) {};
\node at (393,446) [vertex,fill=black,label=left:{$w_5$}] (v11) {};
\node at (347,294) [vertex,fill=white,label=left:{$w_0$}] (v12) {};
\node at (396,295) [vertex,fill=black,label=right:{$w_1$}] (v13) {};
\node at (435,325) [vertex,fill=white,label=below right:{$w_2$}] (v14) {};
\node at (307,322) [vertex,fill=black,label=below left:{$w_9$}] (v15) {};
\node at (291,368) [vertex,fill=white,label=above left:{$w_8$}] (v16) {};
\node at (305,415) [vertex,fill=black,label=above left:{$w_7$}] (v17) {};
\node at (344,445) [vertex,fill=white,label=left:{$w_6$}] (v18) {};
\node at (433,418) [vertex,fill=white,label=above right:{$w_4$}] (v19) {};
\node at (449,372) [vertex,fill=black,label=above right:{$w_3$}] (v20) {};
\node at (336,334) (v21) {};
\node at (404,402) (v22) {};
\node at (322,360) (v23) {};
\node at (327,346) (v24) {};
\node at (362,320) (v25) {};
\node at (417,375) (v26) {};
\node at (412,389) (v27) {};
\node at (377,320) (v28) {};
\node at (336,402) (v29) {};
\node at (362,415) (v30) {};
\node at (377,415) (v31) {};
\node at (392,325) (v32) {};
\node at (348,325) (v33) {};
\node at (404,334) (v34) {};
\node at (417,360) (v35) {};
\node at (322,375) (v36) {};
\node at (327,389) (v37) {};
\node at (348,410) (v38) {};
\node at (392,410) (v39) {};
\node at (413,346) (v40) {};
\path
	(v1) edge (v2)
	(v2) edge (v3)
	(v3) edge (v4)
	(v4) edge (v5)
	(v5) edge (v6)
	(v6) edge (v7)
	(v7) edge (v8)
	(v8) edge (v9)
	(v9) edge (v10)
	(v10) edge (v1)
	(v1) edge (v12)
	(v2) edge (v13)
	(v3) edge (v14)
	(v4) edge (v20)
	(v5) edge (v19)
	(v6) edge (v11)
	(v7) edge (v18)
	(v8) edge (v17)
	(v9) edge (v16)
	(v10) edge (v15)
	(v15) edge (v24)
	(v15) edge (v21)
	(v12) edge (v33)
	(v12) edge (v25)
	(v13) edge (v28)
	(v13) edge (v32)
	(v14) edge (v34)
	(v14) edge (v40)
	(v20) edge (v35)
	(v20) edge (v26)
	(v19) edge (v27)
	(v19) edge (v22)
	(v11) edge (v39)
	(v11) edge (v31)
	(v18) edge (v30)
	(v18) edge (v38)
	(v17) edge (v29)
	(v17) edge (v37)
	(v16) edge (v36)
	(v16) edge (v23)
	;
\end{tikzpicture} & 
        \begin{tikzpicture}[x=0.2mm,y=-0.2mm,inner sep=0.2mm,scale=1,thick,vertex/.style={circle,draw,minimum size=8,fill=lightgray}]
\node at (324,227) [vertex,fill=black,label=above left:{$v_0$}] (v1) {};
\node at (416,227) [vertex,fill=white,label=above right:{$v_1$}] (v2) {};
\node at (491,282) [vertex,fill=black,label=above right:{$v_2$}] (v3) {};
\node at (520,370) [vertex,fill=white,label=right:{$v_3$}] (v4) {};
\node at (491,458) [vertex,fill=black,label=below right:{$v_4$}] (v5) {};
\node at (416,513) [vertex,fill=white,label=below right:{$v_5$}] (v6) {};
\node at (324,513) [vertex,fill=black,label=below left:{$v_6$}] (v7) {};
\node at (249,458) [vertex,fill=white,label=below left:{$v_7$}] (v8) {};
\node at (220,370) [vertex,fill=black,label=left:{$v_8$}] (v9) {};
\node at (249,282) [vertex,fill=white,label=above left:{$v_9$}] (v10) {};
\node at (393,446) [vertex,fill=black,label=left:{$w_5$}] (v11) {};
\node at (347,294) [vertex,fill=white,label=left:{$w_0$}] (v12) {};
\node at (396,295) [vertex,fill=black,label=right:{$w_1$}] (v13) {};
\node at (435,325) [vertex,fill=white,label=below right:{$w_2$}] (v14) {};
\node at (307,322) [vertex,fill=black,label=below left:{$w_9$}] (v15) {};
\node at (291,368) [vertex,fill=white,label=above left:{$w_8$}] (v16) {};
\node at (305,415) [vertex,fill=black,label=above left:{$w_7$}] (v17) {};
\node at (344,445) [vertex,fill=white,label=left:{$w_6$}] (v18) {};
\node at (433,418) [vertex,fill=white,label=above right:{$w_4$}] (v19) {};
\node at (449,372) [vertex,fill=black,label=above right:{$w_3$}] (v20) {};
\node at (336,334) (v21) {};
\node at (404,402) (v22) {};
\node at (322,360) (v23) {};
\node at (327,346) (v24) {};
\node at (362,320) (v25) {};
\node at (417,375) (v26) {};
\node at (412,389) (v27) {};
\node at (377,320) (v28) {};
\node at (336,402) (v29) {};
\node at (362,415) (v30) {};
\node at (377,415) (v31) {};
\node at (392,325) (v32) {};
\node at (348,325) (v33) {};
\node at (404,334) (v34) {};
\node at (417,360) (v35) {};
\node at (322,375) (v36) {};
\node at (327,389) (v37) {};
\node at (348,410) (v38) {};
\node at (392,410) (v39) {};
\node at (413,346) (v40) {};
\path
	(v1) edge (v2)
	(v3) edge (v4)
	(v5) edge (v6)
	(v7) edge (v8)
	(v9) edge (v10)
	(v1) edge (v12)
	(v2) edge (v13)
	(v3) edge (v14)
	(v4) edge (v20)
	(v5) edge (v19)
	(v6) edge (v11)
	(v7) edge (v18)
	(v8) edge (v17)
	(v9) edge (v16)
	(v10) edge (v15)
	(v15) edge (v24)
	(v15) edge (v21)
	(v12) edge (v33)
	(v12) edge (v25)
	(v13) edge (v28)
	(v13) edge (v32)
	(v14) edge (v34)
	(v14) edge (v40)
	(v20) edge (v35)
	(v20) edge (v26)
	(v19) edge (v27)
	(v19) edge (v22)
	(v11) edge (v39)
	(v11) edge (v31)
	(v18) edge (v30)
	(v18) edge (v38)
	(v17) edge (v29)
	(v17) edge (v37)
	(v16) edge (v36)
	(v16) edge (v23)
	;
\end{tikzpicture} \\
        $G$ & $G'$ \\
        \multicolumn{2}{c}{\begin{tikzpicture}[x=0.2mm,y=-0.2mm,inner sep=0.2mm,scale=1,thick,vertex/.style={circle,draw,minimum size=8,fill=lightgray}]
\node at (324,227) [vertex,fill=black,label=above left:{$u_3$}] (v1) {};
\node at (416,227) [vertex,fill=white,label=above right:{$u_4$}] (v2) {};
\node at (491,282) [vertex,fill=black,label=above right:{$u_7$}] (v3) {};
\node at (520,370) [vertex,fill=white,label=right:{$u_8$}] (v4) {};
\node at (491,458) [vertex,fill=black,label=below right:{$u_1$}] (v5) {};
\node at (416,513) [vertex,fill=white,label=below right:{$u_2$}] (v6) {};
\node at (324,513) [vertex,fill=black,label=below left:{$u_5$}] (v7) {};
\node at (249,458) [vertex,fill=white,label=below left:{$u_6$}] (v8) {};
\node at (220,370) [vertex,fill=black,label=left:{$u_9$}] (v9) {};
\node at (249,282) [vertex,fill=white,label=above left:{$u_0$}] (v10) {};
\node at (393,446) [vertex,fill=black,label=left:{$v_2$}] (v11) {};
\node at (347,294) [vertex,fill=white,label=left:{$v_3$}] (v12) {};
\node at (396,295) [vertex,fill=black,label=right:{$v_4$}] (v13) {};
\node at (435,325) [vertex,fill=white,label=below right:{$v_7$}] (v14) {};
\node at (307,322) [vertex,fill=black,label=below left:{$v_0$}] (v15) {};
\node at (291,368) [vertex,fill=white,label=above left:{$v_9$}] (v16) {};
\node at (305,415) [vertex,fill=black,label=above left:{$v_6$}] (v17) {};
\node at (344,445) [vertex,fill=white,label=left:{$v_5$}] (v18) {};
\node at (433,418) [vertex,fill=white,label=above right:{$v_1$}] (v19) {};
\node at (449,372) [vertex,fill=black,label=above right:{$v_8$}] (v20) {};
\path
	(v1) edge (v2)
	(v2) edge (v3)
	(v3) edge (v4)
	(v4) edge (v5)
	(v5) edge (v6)
	(v6) edge (v7)
	(v7) edge (v8)
	(v8) edge (v9)
	(v9) edge (v10)
	(v10) edge (v1)
	(v1) edge (v12)
	(v2) edge (v13)
	(v3) edge (v14)
	(v4) edge (v20)
	(v5) edge (v19)
	(v6) edge (v11)
	(v7) edge (v18)
	(v8) edge (v17)
	(v9) edge (v16)
	(v10) edge (v15)
	;
\end{tikzpicture}} \\
        \multicolumn{2}{c}{$H$}
    \end{tabular}
    \caption{Graphs from the proof of Proposition~\ref{vplus10}}
    \label{fig:girth8}
\end{figure}
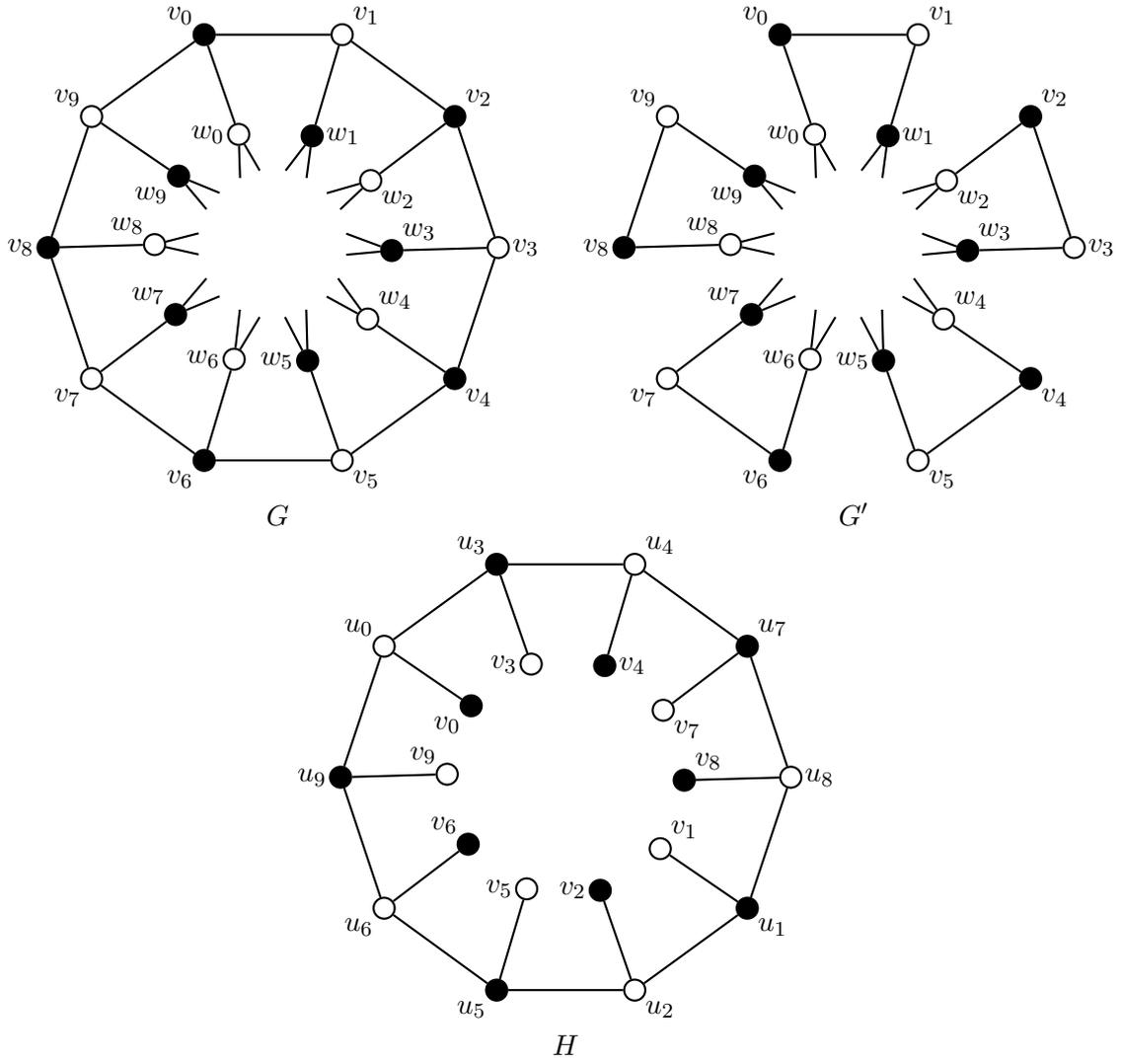
Let $C$ be a cycle of $G^*$.  If $C$ contains only edges of $G'$ or only edges of $H$, then $C$ has length at least 8. If not, then $C$ must contain at least three edges from $H$ including two edges $\{u_m,v_m\}$ and $\{u_n,v_n\}$, $0 \leq m,n \leq 9$. We need to show that it now contains at least five edges from $G'$. If $C$ contains only one edge from $G'$ it would have to be $\{v_{2i},v_{2i+1}\}$ for some $i$ such that $0 \leq i \leq 4$. But then $C$ would actually contain seven edges from $H$ and so have length 8. Clearly $C$ cannot contain only two edges from $G'$. If it contains just three edges then these would have to be $\{v_i,w_i\}$, $\{w_i,w_{i+5}\}$ and $\{w_{i+5},v_{i+5}\}$ for some $i$ such that $0 \leq i \leq 4$ and then $C$ would contain five edges from $H$ and so again have length 8. The remaining possibility is that $C$ contains just four edges from $G'$, in which case vertices $v_m$ and $v_n$ are in the same partition and their distance in $H$ is at least 4.
\end{proof}
We can now prove the main result of this section.
\begin{theorem}\label{avoidtri}
There exists a symmetric configuration $v_3$ with no triangles if and only if  $v=15$ or $v \geq 17$. 
\end{theorem}
\begin{proof}
The smallest cubic graph of girth 8 has order 30 and is Tutte's 8-cage. Thus there are no configurations $v_3$ with no triangles for $7 \leq v \leq 14$. The graph is bipartite, unique, contains a 10-cycle and is the Levi graph of the Cremona-Richmond configuration or generalised quadrangle GQ($2,2$). It has an elegant construction as follows. Let $S$ be a set of cardinality 6. Define the points of the configuration to be the set of unordered pairs of elements of $S$ (i.e. subsets of cardinality 2) and the lines to be the partitions of $S$ into disjoint pairs. An exhaustive computer search shows that there is no cubic bipartite graph of order 32 and girth 8, a fact confirmed on page 734 of~\cite{R}.

From Proposition~\ref{vplus10} in order to complete the proof of the theorem, it suffices to exhibit cubic bipartite graphs of girth 8 also containing a 10-cycle of orders 34, 36, 38 and 42. There is a unique graph of order 34 and girth 8 and three graphs of order 36 and girth 8, again see page 734 of~\cite{R}. We have determined that all these four graphs are bipartite and contain a 10-cycle. The corresponding configurations are given below.

\vspace*{2ex}
\begin{config}
\-\ 012 034 056 178 19a 2bc 2de 37b 39d 48e 4af 58c 5df 6ab 6eg 7fg 9cg

\-\ 012 034 056 178 19a 2bc 2de 37b 39d 48c 4af 58g 5ae 6ch 6df 7eh 9gh bfg

\-\ 012 034 056 178 19a 2bc 2de 37b 39d 48c 4fg 58e 59f 6ch 6dg 7fh abg aeh

\-\ 012 034 056 178 19a 2bc 2de 37b 39d 48e 4af 57g 5df 68c 69h acg bfh egh
\end{config}

\vspace*{1ex}
Configurations corresponding to bipartite graphs of orders 38 and 42 of girth 8 and containing a 10-cycle are also given below.

\vspace*{2ex}
\begin{config}
\-\ 012 034 056 178 19a 2bc 2de 37b 39d 48c 4af 57e 5ag 68h 69i bfh cgi dgh efi

\-\ 012 034 056 178 19a 2bc 2de 37b 39d 48c 4af 57e 5gh 68i 6fj 9gk ahi bfg cjk dhj eik
\end{config}

This completes the proof of the theorem.
\end{proof}

% ================================================================
\section{Configurations with many triangles}\label{sec:manytri}
For any symmetric configuration $v_3$, say $C$, we let $t(C)$ denote the number of triangles in $C$. For any $v\geq 7$ we let $T(v)$ denote the maximum value of $t(C)$ for any configuration $v_3$.
\begin{theorem}\label{thm:3v}
For any $v\geq 8$, $T(v)\leq 3v$.
\end{theorem}
\begin{proof}
As noted above, the number of triangles in a configuration is precisely the number of 6-cycles in its Levi graph. We therefore count the largest possible number of 6-cycles in a cubic bipartite graph of girth 6. To do this it is most convenient to count the number of 6-cycles through any given edge.

Referring to Figure~\ref{fig:edh}, the number of 6-cycles through the edge $ab$ is equal to the number of edges between those vertices at distance 2 from $ab$, i.e. the 8 vertices between the dotted lines. There are a maximum of 8 such edges; but if all 8 are present then the graph is the Heawood graph, which is the Levi graph of the unique configuration $7_3$.

So suppose that the maximum number of 6-cycles through any vertex is 7, as in the figure. Then the 14 vertices at maximum distance 2 from $ab$ form a Heawood graph with a single edge deleted; since the Heawood graph is edge-transitive we may assume without loss that the missing edge is $uv$ as in the figure. Then the subgraph $H$ induced by these 14 vertices is joined to the rest of the graph by edges $ux$ and $vy$. Now $H$ contains exactly 20 6-cycles, since a Heawood graph has 28 and removing the edge $uv$ reduces this by 8. Thus the average number of 6-cycles through any edge in $H$ is 6. Moreover, since $u$ and $v$ are at distance 5 in $H$, there are no 6-cycles through either edge $ux$ or $vy$.

The graph may contain multiple edges lying in 7 6-cycles, so there may be more than one of these induced subgraphs isomorphic to $H$, joined to the remainder of the graph by edges which are in no 6-cycle. The remainder of the edges in the graph not in any of these induced subgraphs must lie in a maximum of 6 6-cycles, and this includes a non-zero number of edges not in any 6-cycle at all. Thus if there is at least one edge lying in 7 6-cycles, the average number of 6-cycles through all the edges in the graph is strictly less than 6. Since there are $3v$ edges in the Levi graph and each 6-cycle goes through 6 edges, it follows that there are strictly fewer than $3v$ 6-cycles in the Levi graph and hence triangles in the configuration.

The remaining possibility to achieve equality in the bound is that there are precisely 6 6-cycles through each edge in the Levi graph. This situation is attained by the unique configuration $8_3$.
\end{proof}
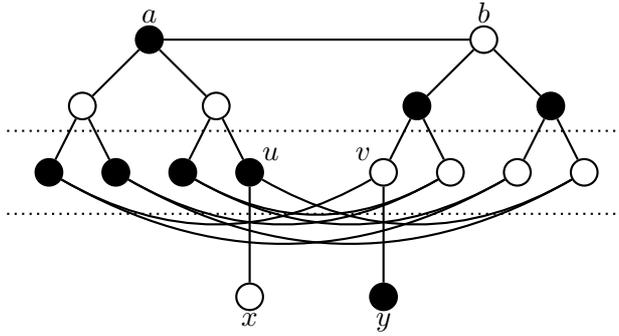
\begin{figure}
    \centering
    \begin{tikzpicture}[x=0.2mm,y=-0.2mm,inner sep=0.2mm,scale=0.55,thick,vertex/.style={circle,draw,minimum size=10,fill=lightgray}]
\node at (170,256) [vertex,fill=black,label=above:{$a$}] (v13) {};
\node at (570,256) [vertex,fill=white,label=above:{$b$}] (v12) {};
\node at (490,336) [vertex,fill=black] (v1) {};
\node at (90,336) [vertex,fill=white] (v4) {};
\node at (250,336) [vertex,fill=white] (v8) {};
\node at (650,336) [vertex,fill=black] (v11) {};
\node at (50,416) [vertex,fill=black] (v5) {};
\node at (130,416) [vertex,fill=black] (v9) {};
\node at (210,416) [vertex,fill=black] (v7) {};
\node at (290,416) [vertex,fill=black,label=above right:{$u$}] (v3) {};
\node at (450,416) [vertex,fill=white,label=above left:{$v$}] (v2) {};
\node at (530,416) [vertex,fill=white] (v14) {};
\node at (610,416) [vertex,fill=white] (v6) {};
\node at (690,416) [vertex,fill=white] (v10) {};
\node at (290,566) [vertex,fill=white,label=below:{$x$}] (v16) {};
\node at (450,566) [vertex,fill=black,label=below:{$y$}] (v15) {};
\draw[dotted] (0,366) -- (740,366);
\draw[dotted] (0,466) -- (740,466);
\path
	(v1) edge (v2)
	(v1) edge (v14)
	(v3) edge (v8)
	(v4) edge (v5)
	(v4) edge (v13)
	(v5) edge[bend right] (v6)
	(v7) edge[bend right] (v6)
	(v7) edge (v8)
	(v9) edge[bend right] (v10)
	(v9) edge[bend right] (v14)
	(v10) edge (v11)
	(v11) edge (v12)
	(v12) edge (v13)
	(v7) edge[bend right] (v14)
	(v5) edge[bend right] (v2)
	(v3) edge[bend right] (v10)
	(v6) edge (v11)
	(v4) edge (v9)
	(v1) edge (v12)
	(v8) edge (v13)
	(v3) edge (v16)
	(v2) edge (v15)
	;
\end{tikzpicture}
    \caption{An edge $ab$ contained in 7 6-cycles}
    \label{fig:edh}
\end{figure}
We note that the proof of Theorem~\ref{thm:3v} implies that equality exists in the bound if and only if every edge in the Levi graph lies in exactly 6 6-cycles. Such a graph is called \emph{edge-girth-regular} and these objects were studied in~\cite{J}. The only known edge-girth-regular graph of degree 3, girth 6 and with every edge lying in 6 girth cycles is the M\"obius-Kantor graph, which is the Levi graph of the unique configuration $8_3$. It is an open question whether there exist further configurations where equality is achieved in the bound of Theorem~\ref{thm:3v}.

\begin{theorem}
\[3\geq \limsup_{v\to\infty}\frac{T(v)}{v}\geq \liminf_{v\to\infty}\frac{T(v)}{v}\geq\frac{20}{7}.\]
\end{theorem}
\begin{proof}
The upper bound for the $\limsup$ follows immediately from Theorem~\ref{thm:3v}. To demonstrate the lower bound, we construct an infinite family of configurations on $7n$ points, $n\geq 2$, with exactly $20n$ triangles. To do this we take $n$ copies $H_1,H_2,\ldots,H_n$ of the edge-deleted Heawood graph $H$ from the proof of Theorem~\ref{thm:3v} and connect them via their joining edges in a cyclic manner. The construction is illustrated in Figure~\ref{fig:v21} for the case $n=3$. Since the joining edges do not lie in any 6-cycles, the resulting graph is the Levi graph of a configuration with $7n$ points and $20n$ triangles as required. To complete the proof we need to show that for any $\epsilon>0$, there is some $v_0$ such that for any $v\geq v_0$ we can construct a configuration $C$ on $v$ points with $t(C)/v\geq 20/7-\epsilon$. It is easily seen that we can do this by constructing a graph with a large number $n$ of copies of $H$ plus one edge-deleted Levi graph of a configuration on $7,8,9,10,11,12$ or $13$ points, and join them cyclically as before. We can choose $n$ large enough so that the number of triangles in the resulting configuration $v_3$ will be arbitrarily close to $20v/7$.
\end{proof}
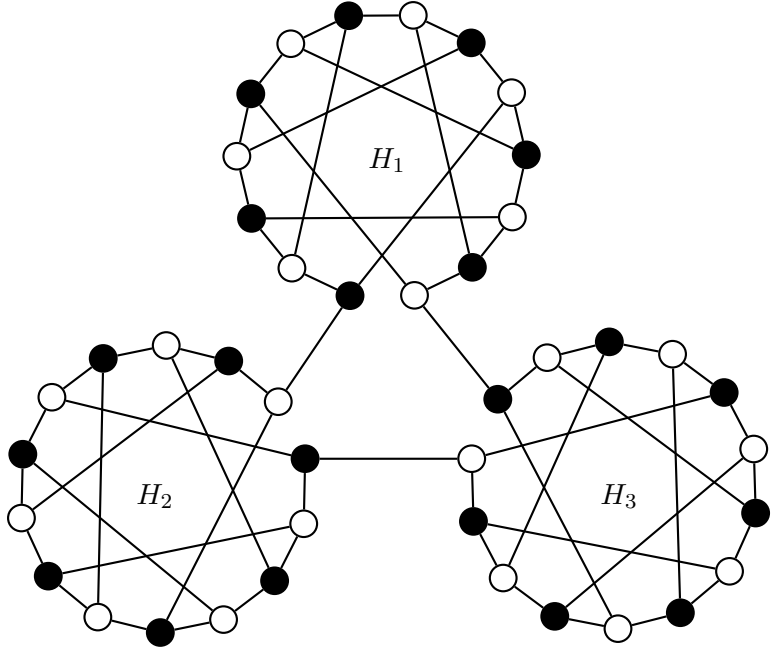
\begin{figure}
    \centering
    \begin{tikzpicture}[x=0.3mm,y=-0.3mm,inner sep=0.2mm,scale=0.55,thick,wvertex/.style={circle,draw,minimum size=10,fill=white},bvertex/.style={circle,draw,minimum size=10,fill=black}]
\node at (385,230) (H1) {$H_1$};
\node at (200,500) (H2) {$H_2$};
\node at (570,500) (H3) {$H_3$};
\node at (204.3,609.7) [bvertex] (v1) {};
\node at (154.5,597.3) [wvertex] (v2) {};
\node at (115,564.4) [bvertex] (v3) {};
\node at (209,379) [wvertex] (v4) {};
\node at (158.8,389.4) [bvertex] (v5) {};
\node at (117.9,420.6) [wvertex] (v6) {};
\node at (94.7,466.4) [bvertex] (v7) {};
\node at (93.6,517.7) [wvertex] (v8) {};
\node at (295.4,568.1) [bvertex] (v9) {};
\node at (318.6,522.3) [wvertex] (v10) {};
\node at (319.7,470) [bvertex] (v11) {};
\node at (298.3,424.2) [wvertex] (v12) {};
\node at (258.8,391.5) [bvertex] (v13) {};
\node at (254.7,599.3) [wvertex] (v14) {};
\node at (453,316.2) [bvertex] (v15) {};
\node at (406.9,338.7) [wvertex] (v16) {};
\node at (355.5,339) [bvertex] (v17) {};
\node at (308.3,136.5) [wvertex] (v18) {};
\node at (276.4,176.8) [bvertex] (v19) {};
\node at (265.3,226.9) [wvertex] (v20) {};
\node at (277,276.9) [bvertex] (v21) {};
\node at (309.2,316.9) [wvertex] (v22) {};
\node at (496,225.8) [bvertex] (v23) {};
\node at (484.3,175.8) [wvertex] (v24) {};
\node at (452.1,135.8) [bvertex] (v25) {};
\node at (405.9,113.7) [wvertex] (v26) {};
\node at (354.4,114) [bvertex] (v27) {};
\node at (484.9,275.8) [wvertex] (v28) {};
\node at (678.8,513.7) [bvertex] (v29) {};
\node at (658,560.5) [wvertex] (v30) {};
\node at (618.8,593.8) [bvertex] (v31) {};
\node at (452.5,470) [wvertex] (v32) {};
\node at (453.9,520.4) [bvertex] (v33) {};
\node at (477.6,565.9) [wvertex] (v34) {};
\node at (518.7,596.7) [bvertex] (v35) {};
\node at (569.1,606.6) [wvertex] (v36) {};
\node at (653.7,416.8) [bvertex] (v37) {};
\node at (612.6,386) [wvertex] (v38) {};
\node at (562.2,376.1) [bvertex] (v39) {};
\node at (512.6,388.8) [wvertex] (v40) {};
\node at (473.3,422.2) [bvertex] (v41) {};
\node at (677.3,462.2) [wvertex] (v42) {};
\path
	(v1) edge (v2)
	(v1) edge (v14)
	(v2) edge (v3)
	(v3) edge (v8)
	(v4) edge (v5)
	(v4) edge (v13)
	(v5) edge (v6)
	(v6) edge (v7)
	(v7) edge (v8)
	(v9) edge (v10)
	(v9) edge (v14)
	(v10) edge (v11)
	(v12) edge (v13)
	(v7) edge (v14)
	(v2) edge (v5)
	(v3) edge (v10)
	(v6) edge (v11)
	(v4) edge (v9)
	(v1) edge (v12)
	(v8) edge (v13)
	(v15) edge (v16)
	(v15) edge (v28)
	(v17) edge (v22)
	(v18) edge (v19)
	(v18) edge (v27)
	(v19) edge (v20)
	(v20) edge (v21)
	(v21) edge (v22)
	(v23) edge (v24)
	(v23) edge (v28)
	(v24) edge (v25)
	(v25) edge (v26)
	(v26) edge (v27)
	(v21) edge (v28)
	(v16) edge (v19)
	(v17) edge (v24)
	(v20) edge (v25)
	(v18) edge (v23)
	(v15) edge (v26)
	(v22) edge (v27)
	(v29) edge (v30)
	(v29) edge (v42)
	(v30) edge (v31)
	(v31) edge (v36)
	(v32) edge (v33)
	(v33) edge (v34)
	(v34) edge (v35)
	(v35) edge (v36)
	(v37) edge (v38)
	(v37) edge (v42)
	(v38) edge (v39)
	(v39) edge (v40)
	(v40) edge (v41)
	(v35) edge (v42)
	(v30) edge (v33)
	(v31) edge (v38)
	(v34) edge (v39)
	(v32) edge (v37)
	(v29) edge (v40)
	(v36) edge (v41)
	(v16) edge (v41)
	(v11) edge (v32)
	(v12) edge (v17)
	;
\end{tikzpicture}
    \caption{The Levi graph of a configuration on 21 points with 60 triangles}
    \label{fig:v21}
\end{figure}%

\newpage
We note that the graph depicted in Figure~\ref{fig:v21} has appeared before in the literature. For example, in previous works of the present authors it arises in connection with colouring problems~\cite{EGS1} and embeddings~\cite{EGS2}.

% ================================================================
\section{Cyclic configurations}\label{sec:cyclic}
A symmetric configuration $v_3$ is cyclic if it admits an automorphism of order $v$. Such a configuration can be realised on the set $\{0,1,\ldots, v-1\}$ as a single orbit under the mapping $i \mapsto i+1$ (mod $v$) of a starter block $\{0,a,a+b\},~a,b\in \{1,2,\ldots,v-2\},~a+b \leq v-1$. The orbit can be described as a cyclically ordered triple $\langle a,b,c \rangle$ where $a+b+c=v$. However in order to generate a configuration, certain requirements need to be met: $1 \leq a,b,c \leq v-3,~a \neq b \neq c \neq a$ and if $v$ is even, $a,b,c \neq v/2$. Further, for the configuration to be connected we must have $\gcd(a,b,c) =1$. We are able to give a complete analysis of the occurrence of triangles in these configurations.

\begin{theorem}\label{cyctri}
Let $\mathcal{C}$ be a cyclic symmetric configuration $v_3$. Then the number of triangles contained in $\mathcal{C}$ is (i) $4v$ if $v=7$, (ii) $3v$ if $v=8$, (iii) $7v/3$ if $v=9$ and (iv) $v, 4v/3$ or $2v$ if $v \geq 10$.
\end{theorem}
\begin{proof}
The configurations $7_3$ and $8_3$ are unique up to isomorphism and are cyclic. It is easily determined that the number of triangles they contain are 28 and 24 respectively. Of the three non-isomorphic configurations $9_3$, only one is cyclic and again it is easily determined that the number of triangles is 21.

Now let $v \geq 10$. The incidence graph is a Cayley graph Cay($\mathbb{Z}_v, S$) where the connection set $S = \{\pm a, \pm b, \pm c\}$. Triangles in the configuration will occur as 3-cycles in the incidence graph. We count these by identifying edges which sum to zero, noting that $(x,y,z)$, $(y,z,x)$, $(z,x,y)$, $(-x,-z,-y)$, $(-z,-y,-x)$ and $(-y,-x,-z)$ all count the same triangles. First consider the two sets of edges $(a,b,c)$ and $(a,c,b)$; the first of these correspond to blocks of the configuration and the second to $v$ triangles in the configuration.

There are two further situations where triangles occur. The first of these is where, without loss of generality, either $b=2a$ or $c=2a$, i.e. the cyclically ordered triple $\langle a,b,c \rangle$ is either $\langle a,2a,v-3a \rangle$. or $\langle a,v-3a,2a \rangle$. (Note that $b=v-2a$ and $c=v-2a$ are not possible.) In this case there are a further $v$ triangles corresponding to edges $(a,a,-2a)$ in the incidence graph. Potentially an orbit of this form may generate further triangles from 3-cycles as follows; $(2a,2a,3a)$, $(3a,3a,a)$ and $(3a,3a,2a)$. If $\gcd(a,v)=1$, the first two can only appear when $v=7$ and the third when $v=8$. This accounts for the ``extra'' triangles in these configurations noted above. If $\gcd(a,v) > 1$, the configurations are disconnected.

The second situation can only occur when $v \equiv 0$ (mod 3) and is without loss of generality when $a = v/3$ or $a = 2v/3$. There are $v/3$ triangles $\langle v/3,v/3,v/3 \rangle$ in the configuration corresponding to edges $(v/3,v/3,v/3)$ in the incidence graph. It remains to consider whether both of these two situations can occur simultaneously. There are six possibilities: $(v/3,v/6,v/2)$, $(v/3,v/2,v/6)$, $(v/3,2v/9,4v/9)$, $(v/3,4v/9,2v/9)$, $(2v/3,v/9,2v/9)$ and $(2v/3,2v/9,v/9)$. The first two can only occur when ${v=6s,s>1}$ and the configurations will be disconnected. The other four occur when ${v=9s,s \geq 1}$. When $s=1$, again this accounts for the ``extra'' triangles in the cyclic $9_3$ configuration. When $s > 1$, the configurations are disconnected.  
\end{proof}

% ================================================================
\section{Conclusion}\label{sec:conclu}
In this paper we have begun the study of so-called fragments in symmetric configurations with block size 3. This follows from similar work with respect to Steiner triple systems which was published nearly 30 years ago~\cite{GGM}. An entire chapter of~\cite{CR} is devoted to this and many other aspects of this work due to various authors. We are of the opinion that there is still much which can be done in this area relating to configurations and below we indicate some possible avenues for future research.

As we have shown, the smallest fragment which can be avoided in an infinite number of symmetric configurations is the triangle. Such a configuration on 15 points, the Cremona-Richmond configuration, was already known and we have shown that there exist such configurations $v_3$ for all $ v \geq 17$.  We have also considered symmetric configurations containing ``many'' triangles. A further aspect of this work is that of decomposition. It is relevant to ask about symmetric configurations whose blocks can be decomposed into triangles (with one or two extra blocks if $v$ is not divisible by 3).
Indeed the whole question of the decomposition of configurations $v_3$ into \emph{any} given fragment appears to be open.

The work in this paper may also be extended to non-symmetric configurations $(v_r,b_3)$. Elementary calculations give $b=rv/3$, $a_2=r(r-1)v/2$ and $a_1=rv(rv-9r+6)$ for the numbers of blocks, intersecting triples and disjoint triples respectively; thus an additional parameter $r$, the replication number, is introduced into the equations. Even further, configurations $(v_r,b_3)_\lambda$ where any pair of distinct points is contained in at most $\lambda$ blocks, can be considered. Those configurations with $\lambda = 2$ are called \emph{spatial}. Symmetric spatial configurations are the subject of a paper by Gropp~\cite{Sp} and even though this class of configurations is quite restrictive, results would still be of interest. Two further two-block fragments, $A_3$, abc, abd and $A_4$, abc, abc (repeated triple) may also occur; so, unlike the case where $\lambda = 1$, two-block fragments would not be constant.

However perhaps the most obvious investigation would be to extend the study of fragments in symmetric configurations $v_3$ to those with four blocks. There are 16 four-block fragments which are illustrated on page 210 of~\cite{CR} together with their standard labelling. One of these $C_7$, abc, ade, afg, ahi (4-star) cannot occur. Nevertheless a complete analysis of the number of occurrences of the remaining 15 fragments would seem to be a lengthy and perhaps tedious although interesting undertaking. But a subproject could be to consider only those symmetric configurations which have no triangles, i.e. those whose Levi graph has girth at least 8. This would eliminate $C_6$, $C_8$, $C_{11}$, $C_{12}$, $C_{14}$, $C_{15}$ and $C_{16}$ from the configuration leaving just 8 fragments to consider which is a much more feasible investigation. Within this scenario an 8-cycle in the Levi graph corresponds to a square, i.e. fragment $C_{10}$, abx, bcy, cdz, daw. Thus the construction of cubic bipartite graphs of girth greater than or equal to 10 would give configurations $v_3$ avoiding $C_{10}$ and all fragments containing triangles. A result analogous to Proposition~\ref{vplus10} would go a long way to establishing this.
There is still much to investigate and we hope that the reader  will be inspired to consider some of these ideas.    

% ================================================================
%\newpage
\section*{Appendix}
\begin{table}[h!]
\begin{minipage}[t]{0.2\textwidth}
\vspace{0pt}
\begin{tabular}{|r|r|r|}
	\hline 
	$v$ & $t$ & Count \\ 
	\hline 
	7 & 28 & 1 \\
	\hline
	8 & 24 & 1 \\
	\hline
	9 & 18 & 1 \\
	9 & 20 & 1 \\
	9 & 21 & 1 \\
	\hline
	10 & 17 & 2 \\
	10 & 18 & 3 \\
	10 & 19 & 2 \\
	10 & 20 & 3 \\
	\hline
	11 & 15 & 1 \\
	11 & 16 & 10 \\
	11 & 17 & 7 \\
	11 & 18 & 7 \\
	11 & 19 & 3 \\
	11 & 20 & 1 \\
	11 & 21 & 1 \\
	11 & 22 & 1 \\
	\hline
	12 & 12 & 1 \\
	12 & 13 & 8 \\
	12 & 14 & 22 \\
	12 & 15 & 48 \\
	12 & 16 & 60 \\
	12 & 17 & 41 \\
	12 & 18 & 24 \\
	12 & 19 & 14 \\
	12 & 20 & 5 \\
	12 & 21 & 3 \\
	12 & 22 & 1 \\
	12 & 24 & 2 \\
	\hline
	13 & 9 & 1 \\
	13 & 10 & 2 \\
	13 & 11 & 12 \\
	13 & 12 & 67 \\
	13 & 13 & 190 \\
	13 & 14 & 371 \\
	13 & 15 & 418 \\
	\hline
\end{tabular}
\end{minipage} \hfill
\begin{minipage}[t]{0.2\textwidth}
\vspace{0pt}
\begin{tabular}{|r|r|r|}
	\hline 
	$v$ & $t$ & Count \\ 
	\hline 
	13 & 16 & 409 \\
	13 & 17 & 265 \\
	13 & 18 & 156 \\
	13 & 19 & 74 \\
	13 & 20 & 37 \\
	13 & 21 & 14 \\
	13 & 22 & 9 \\
	13 & 23 & 4 \\
	13 & 24 & 3 \\
	13 & 25 & 1 \\
	13 & 26 & 1 \\
	13 & 28 & 1 \\
	13 & 32 & 1 \\
	\hline
	14 & 6 & 1 \\
	14 & 8 & 5 \\
	14 & 9 & 24 \\
	14 & 10 & 145 \\
	14 & 11 & 521 \\
	14 & 12 & 1512 \\
	14 & 13 & 2901 \\
	14 & 14 & 4086 \\
	14 & 15 & 4121 \\
	14 & 16 & 3247 \\
	14 & 17 & 2236 \\
	14 & 18 & 1304 \\
	14 & 19 & 640 \\
	14 & 20 & 335 \\
	14 & 21 & 159 \\
	14 & 22 & 69 \\
	14 & 23 & 33 \\
	14 & 24 & 36 \\
	14 & 25 & 5 \\
	14 & 26 & 2 \\
	14 & 27 & 4 \\
	14 & 28 & 6 \\
	14 & 31 & 2 \\
	\hline
\end{tabular}
\end{minipage} \hfill
\begin{minipage}[t]{0.2\textwidth}
\vspace{0pt}
\begin{tabular}{|r|r|r|}
	\hline 
	$v$ & $t$ & Count \\ 
	\hline 
	14 & 32 & 3 \\
	14 & 40 & 1 \\
	\hline
	15 & 0 & 1 \\
	15 & 4 & 3 \\
	15 & 6 & 12 \\
	15 & 7 & 40 \\
	15 & 8 & 254 \\
	15 & 9 & 1129 \\
	15 & 10 & 4252 \\
	15 & 11 & 11877 \\
	15 & 12 & 24510 \\
	15 & 13 & 38017 \\
	15 & 14 & 44834 \\
	15 & 15 & 41585 \\
	15 & 16 & 32177 \\
	15 & 17 & 20914 \\
	15 & 18 & 12585 \\
	15 & 19 & 6517 \\
	15 & 20 & 3341 \\
	15 & 21 & 1617 \\
	15 & 22 & 851 \\
	15 & 23 & 404 \\
	15 & 24 & 206 \\
	15 & 25 & 62 \\
	15 & 26 & 54 \\
	15 & 27 & 39 \\
	15 & 28 & 23 \\
	15 & 29 & 6 \\
	15 & 30 & 12 \\
	15 & 31 & 9 \\
	15 & 32 & 6 \\
	15 & 33 & 1 \\
	15 & 38 & 2 \\
	15 & 40 & 1 \\
	\hline
	16 & 3 & 1 \\
	16 & 4 & 5 \\
	\hline
\end{tabular}
\end{minipage} \hfill
\begin{minipage}[t]{0.2\textwidth}
\vspace{0pt}
\begin{tabular}{|r|r|r|}
	\hline 
	$v$ & $t$ & Count \\ 
	\hline 
	16 & 5 & 40 \\
	16 & 6 & 277 \\
	16 & 7 & 1699 \\
	16 & 8 & 8325 \\
	16 & 9 & 31782 \\
	16 & 10 & 92432 \\
	16 & 11 & 206506 \\
	16 & 12 & 357339 \\
	16 & 13 & 480580 \\
	16 & 14 & 517343 \\
	16 & 15 & 458294 \\
	16 & 16 & 344331 \\
	16 & 17 & 225866 \\
	16 & 18 & 133408 \\
	16 & 19 & 72369 \\
	16 & 20 & 37472 \\
	16 & 21 & 18591 \\
	16 & 22 & 9255 \\
	16 & 23 & 4570 \\
	16 & 24 & 2123 \\
	16 & 25 & 871 \\
	16 & 26 & 650 \\
	16 & 27 & 310 \\
	16 & 28 & 169 \\
	16 & 29 & 101 \\
	16 & 30 & 82 \\
	16 & 31 & 43 \\
	16 & 32 & 18 \\
	16 & 33 & 5 \\
	16 & 34 & 1 \\
	16 & 35 & 1 \\
	16 & 36 & 12 \\
	16 & 37 & 4 \\
	16 & 38 & 1 \\
	16 & 39 & 1 \\
	\hline
\end{tabular}
\end{minipage}
\caption{Numbers of triangles $t$ in configurations $v_3$}
\label{tab:triangles}
\end{table}

\section*{Acknowledgements}
The third author acknowledges support from the APVV Research Grants 15-0220 and 17-0428, and the VEGA Research Grants 1/0142/17 and 1/0238/19.

% ================================================================

% ================================================================
\end{document}